\documentclass[req-no]{amsart}

\usepackage{amsmath}
\usepackage{amssymb}
\usepackage{mhequ}
\usepackage{color}
\usepackage{graphicx}

\numberwithin{equation}{section}

\newtheorem{theorem}{Theorem}[section]
\newtheorem{lemma}[theorem]{Lemma}
\newtheorem{proposition}[theorem]{Proposition}

\theoremstyle{definition}
\newtheorem{definition}[theorem]{Definition}

\theoremstyle{remark}

\newcommand{\N} {\mathbb N}
\newcommand{\R} {\mathbb R}
\newcommand{\T} {\mathbb T}

\newcommand{\cA}{\mathcal A}
\newcommand{\cC}{\mathcal C}

\newcommand{\cK}{\mathcal K}
\newcommand{\cM}{\mathcal M}

\newcommand{\dist}{\operatorname{dist}}
\newcommand{\inner}{\operatorname{int}}

\newcommand{\defref}[1]{Definition~\ref{#1}}

\newcommand{\propref}[1]{Proposition~\ref{#1}}

\newcommand{\secref}[1]{Section~\ref{#1}}
\newcommand{\theoref}[1]{Theorem~\ref{#1}}

\newcommand{\fa}{\quad \text{for all }\,}

\newcommand{\mwith}{\text{ with }}
\newcommand{\qandq}{\quad \text{and} \quad }
\newcommand{\qorq}{\quad \text{or} \quad }

\newcommand{\es}{\emptyset}
\newcommand{\sm}{\setminus}
\newcommand{\tm}{\times}

\newcommand{\eps}{\varepsilon}

\newcommand{\opintb}[1]{\big(#1\big)}

\newcommand{\set}[1]{\{#1\}}
\newcommand{\setb}[1]{\big\{#1\big\}}
\newcommand{\setB}[1]{\Big\{#1\Big\}}

\newcommand{\norm}[1]{\|#1\|}

\title[Bifurcations of set-valued dynamical systems]
 {Topological bifurcations\\ of minimal invariant sets\\ for set-valued dynamical systems} 

\author[J.S.W.~Lamb, M.~Rasmussen and C.S.~Rodrigues]{Jeroen S.W.~Lamb, Martin Rasmussen and Christian S.~Rodrigues}

\subjclass[2010]{37G35, 37H20, 37C70, 49K21 (primary), 37B25, 34A60 (secondary)}

\thanks{The first author was supported by a FAPESP-Brazil Visiting Professorship (2009-18338-2). The first and second author gratefully acknowledge partial support by EU IRSES project DynEurBraz and the warm hospitality of IMECC UNICAMP during the development of this paper. The second author was supported by an EPSRC Career Acceleration Fellowship and a Marie Curie Intra European Fellowship of the European Community.\\
We are grateful to Janosch Rieger for many useful comments on a preliminary version of this paper.}

\begin{document}

\maketitle

\date{\today}

\begin{abstract}
  We discuss the dependence of set-valued dynamical systems on parameters. Under mild assumptions which are naturally satisfied for random dynamical systems with bounded noise and control systems, we establish the fact that topological bifurcations of minimal invariant sets are discontinuous with respect to the Hausdorff metric, taking the form of lower semi-continuous explosions and instantaneous appearances. We also characterise these transitions by properties of Morse-like decompositions.
\end{abstract}

\section{Introduction}

Dynamical systems usually refer to time evolutions of states, where each initial condition leads to a unique state of the system in the future.
Set-valued dynamical systems allow a multi-valued future, motivated, for instance, by impreciseness or uncertainty. In particular, set-valued
dynamical systems naturally arise in the context of random and control systems.

The main motivation for the work in this paper is the study of random dynamical systems represented by a mapping $f:\R^d\to \R^d$ with a bounded noise of size $\eps>0$,
\begin{equation}\label{eqn_1}
  x_{n+1} = f(x_n) + \xi_n\,,
\end{equation}
where the sequence $(\xi_n)_{n\in\N}$ is a uniformly distributed random variable with values in $\overline{B_\eps(0)}:= \set{x\in \R^d: \norm{x}\le \eps}$. The collective behavior of all
future
trajectories is then represented by a set-valued mapping $F:\cK(\R^d)\to\cK(\R^d)$, defined by
\begin{displaymath}
  F(M):= \overline{B_\eps(f(M))} \fa M \in \cK(\R^d)\,,
\end{displaymath}
where $\cK(\R^d)$ is the set of all nonempty compact subsets of $\R^d$.

Under the natural assumption that the probability distribution on $\overline{B_\eps(0)}$ has a non-vanishing Lebesgue density, it turns out that the
supports of stationary measures of the random dynamical system are minimal invariant sets of the set-valued mapping~$F$ \cite{Araujo_00_1,Zmarrou_07_1}. A
minimal invariant set is a compact set $M\subset \R^d$ that is invariant (i.e.~$F(M)=M$) and contains no proper invariant subset.

In this paper, we are mainly interested in topological bifurcations of minimal invariant sets, while considering a parameterized family of set-valued mappings
$(F_\lambda)_{\lambda \in \Lambda}$, where $\Lambda$ is a metric space. These bifurcations involve discontinuous changes as well as disappearances of minimal invariant
sets under variation of $\lambda$.

\begin{definition}[Topological bifurcation of minimal invariant sets]\label{def_1}
  Let $(F_\lambda)_{\lambda \in \Lambda}$ be a continuously parameterized family of set-valued mappings on $\R^d$ such that $F_\lambda(x)$ contains a ball for all $\lambda\in\Lambda$ and $x\in\R^d$, and let $\cM_\lambda$ denote the union of minimal invariant sets of $F_\lambda$, $\lambda\in\Lambda$. We say that $F_{\lambda}$ admits a \emph{topological bifurcation
  of minimal invariant sets} at $\lambda= \lambda_\ast$ if for any neighbourhood $V$ of $\lambda_\ast$, there does not exist a family of homeomorphisms $(h_\lambda)_{\lambda\in V}$,
  $h_\lambda:\R^d\to \R^d$, depending continuously on $\lambda$, with the property that $h_\lambda (\cM_\lambda) = \cM_{\lambda_\ast}$ for all $\lambda\in V$.
\end{definition}

The main result concerns the necessity of discontinuous changes of minimal invariant sets at bifurcation points with two possible local scenarios.

\begin{theorem}\label{theo_3}
  Suppose that the family $(F_\lambda)_{\lambda\in \Lambda}$ admits a bifurcation at $\lambda=\lambda_\ast$. Then a minimal invariant set changes discontinuously at
  $\lambda=\lambda_\ast$ in one of the following ways:
  \begin{itemize}
    \item[(i)] it explodes lower semi-continuously at $\lambda_\ast$, or \item[(ii)] it disappears instantaneously at $\lambda_\ast$.
  \end{itemize}
\end{theorem}

A more technical formulation of this result with the precise assumptions can be found in \theoref{theo_1}. In fact, the setting of set-valued dynamical systems in this paper is slightly more general than presented above and includes also continuous-time systems.

Another focus of this paper lies in extending Morse decomposition theory to study bifurcation problems in our context. Recently, Morse decompositions
have been discussed for set-valued dynamical systems \cite{Braga_Unpub_2,Li_07_1,McGehee_92_1}, and we generalize certain fundamental results for
attractors and repellers to complementary invariant sets. The second main result of this paper (\theoref{theo_2}) asserts that at a bifurcation
point, these complementary invariant sets must touch.

In the context of the presented motivation above, we note that the study of random dynamical systems with bounded noise can be separated into a
topological part (involving the mapping $F$) and the evolution of measures. In contrast, the topological part is redundant in the case of unbounded
noise (modelled for instance by Brownian motions), where there is only one minimal invariant set, given by the whole space and supporting a unique
stationary measure.

Initial research on bifurcations in random dynamical systems with unbounded noise started in the 1980s, mainly by Ludwig Arnold and co-workers
\cite{Arnold_98_1,Baxendale_94_1,
Johnson_02_1}. Two types of bifurcation have been distinguished so far: the \emph{phenomenological
bifurcation} (P-bifurcation), concerning qualitative changes in stationary densities, and the \emph{dynamical bifurcation} (D-bifurcation),
concerning the sign change of a Lyapunov exponent, cf.~also \cite{Ashwin_99_1}. To a large extent, however, bifurcations in random dynamical systems remain unexplored.

In modelling, bounded noise is often approximated by unbounded noise with highly localized densities in order to enable the use of stochastic
analysis. In this approximation, topological tools to identify bifurcations are inaccessible, leaving the manifestation of a topological bifurcation
as a cumbersome quantitative and qualitative change of properties of invariant measures.

Our work contributes to the abstract theory of set-valued dynamical systems dating back to the 1960s. Early contributions were motivated mainly by
control theory \cite{Roxin_65_2,Kloeden_78_1}, and later developments include stability and attractor theory
\cite{
Araujo_00_1,
Gruene_01_1,Gruene_02_1,Kloeden_11_3,McGehee_92_1,Roxin_97_1}, Morse decompositions
\cite{Braga_Unpub_2,Li_07_1,McGehee_92_1} and ergodic theory \cite{Artstein_00_1} 

Our results build upon initial piloting studies concerning bifurcations in random dynamical systems with bounded noise
\cite{Botts_Unpub_1,Colonius_08_3,Colonius_10_1,Homburg_06_1,Homburg_10_1,Zmarrou_07_1,Zmarrou_08_1} and control systems
\cite{Colonius_03_1,Colonius_08_2,Colonius_09_1,Gayer_04_1,Gayer_05_1}. In particular, \theoref{theo_3} unifies and generalises observations in \cite{Botts_Unpub_1,Homburg_06_1,Zmarrou_07_1} to higher dimensions and non-invertible (set-valued) systems, while the bifurcation analysis in terms of Morse-like decompositions is a novel perspective.

We finally remark that set-valued dynamical systems appear in the literature also as \emph{closed relations}, \emph{general dynamical systems},
\emph{dispersive systems} or \emph{families of semi-groups}.

\section{Set-valued dynamical systems}

Throughout this paper, we consider the phase space of a set-valued dynamical system to be a compact metric space $(X,d)$. We restrict to the setting of a compact phase space throughout the paper, although some of our results extend naturally to noncompact complete phase spaces.

Let $B_\eps(x_0) = \setb{x\in X: d(x,x_0)<\eps}$ denote the $\eps$-neighbourhood of a point $x_0 \in X$. For arbitrary nonempty sets $A,B\subset X$ and $x\in X$, let
$\dist(x,A) := \inf \setb{d(x,y) : y\in A}$ be the \emph{distance} of $x$ to $A$ and $\dist(A,B) := \sup \setb{\dist(x,B): x\in A}$ be the \emph{Hausdorff semi-distance}
of $A$ and $B$. The \emph{Hausdorff distance} of $A$ and $B$ is then defined by $h(A,B):= \max\setb{\dist(A,B), \dist(B,A)}$.

The set of all nonempty compact subsets of $X$ will be denoted by $\cK(X)$. Equipped with the Hausdorff distance $h$, $\cK(X)$ is also a metric space $(\cK(X), h)$. It is well-known that if $X$ is complete or compact, then $\cK(X)$ is also complete or compact, respectively.
Define for a sequence $(M_n)_{n\in\N}$ of bounded subsets of $X$,
\begin{align*}
  \limsup_{n\to\infty}  M_n &:=\setB{ x \in X: \liminf_{n\to\infty} \dist(x,M_n) =0 } \quad\mbox{ and }\\
  \liminf_{n\to\infty}  M_n &:= \setB{ x \in X: \limsup_{n\to\infty} \dist(x,M_n) =0 }
\end{align*}
(see 
\cite[Definition~1.1.1]{Aubin_90_1}).

In this paper, a \emph{set-valued dynamical system} is understood as a mapping $\Phi: \T \tm X \to \cK(X)$ with time set $\T=\N_0$ (discrete) or $\T=\R_0^+$ (continuous), which fulfills $\Phi(1, X) = X$ and the following properties:
\begin{itemize}
  \item[(H1)] $\Phi$ is \emph{continuous}.
  \item[(H2)] $\Phi(0,\xi) = \set{\xi}$ for all $\xi\in X$.
  \item[(H3)] $\Phi(t+\tau, \xi) = \Phi(t,\Phi(\tau,\xi))$ for all $t,\tau \ge 0$ and $\xi\in X$.
\end{itemize}
Note that in (H2), the extension $\Phi(t,M):=\bigcup_{x\in M} \Phi(t,x)$ for $M\subset X$ was used.

There is a one-to-one correspondence between discrete set-valued dynamical systems and continuous mappings $f:X\to \cK(X)$. On the other
hand, set-valued dynamical systems arise in the context of differential inclusions, which canonically generalize ordinary differential equations to multi-valued vector fields \cite{Aubin_84_1,Deimling_92_1}. Note that the $\eps$-perturbation of a discrete mapping as discussed in the Introduction yields a set-valued dynamical system with \emph{continuous} dependence on $x$.

Associated to every set-valued dynamical system is a so-called dual system.

\begin{definition}[Dual of a set-valued dynamical system]
  Let $\Phi:\T\tm X \to \cK(X)$ be a set-valued dynamical system. Then the \emph{dual} of $\Phi$ is defined by $\Phi^*:\T\tm X \to\cK(X)$, where
  \begin{displaymath}
    \Phi^*(t,\xi) :=\setb{x\in X: \xi\in \Phi(t,x)}\fa (t,\xi)\in \T\tm X\,.
  \end{displaymath}
\end{definition}

Note that in case of an invertible (single-valued) dynamical system, $\Phi^*$ coincides with the system under time reversal.

To see that $\Phi^*$ is well-defined, i.e.~$\Phi^*(t,\xi)\in\cK(X)$, consider for given $(t,\xi)\in \T\tm X$ a sequence $(x_n)_{n\in\N}$ in $\Phi^*(t,\xi)$ converging to $x\in X$. This means that $\xi\in \Phi(t,x_n)$ for all $n\in\N$, and hence, $\xi\in \lim_{n\to\infty}
\Phi(t,x_n)=\Phi(t,x)$ by continuity of $\Phi$. Thus, $x\in \Phi^*(t,\xi)$, which proves that this set belongs to $\cK(X)$. Note that $\Phi(1,X)=X$ implies that the images of $\Phi^*$ are non-empty.

The dual $\Phi^*$ was introduced already in \cite{McGehee_92_1} without formalising its properties. The following proposition shows that indeed $\Phi^*$ fulfills the initial value condition (H2) and the group property (H3), but it can be shown that $\Phi^*$ can be discontinuous.

\begin{proposition}
  The dual system $\Phi^*$ fulfills (H2)--(H3).
\end{proposition}

\begin{proof}
  (H2) One has $\Phi^*(0,\xi) = \setb{x\in X: \xi\in \Phi(0,x)} = \setb{x\in X: \xi\in \set{x}} = \set{\xi}$ for all $\xi\in X$.
  \\
  (H3) It follows that {\allowdisplaybreaks \begin{align*}
    & \Phi^*(t+\tau,\xi)  = \setb{x\in X: \xi\in \Phi(t+\tau,x)}= \setb{x\in X: \xi\in \Phi(\tau, \Phi(t ,x))}
    \\
    &= \setb{x\in X: \exists \,y\in \Phi(t ,x): \xi\in \Phi(\tau, y)} = \setb{x\in X: \exists \,y\in \Phi(t ,x): y\in \Phi^*(\tau, \xi)}
    \\
    &= \setb{x\in X: \Phi(t ,x) \cap \Phi^*(\tau, \xi) \not= \es} = \setb{x\in X: \exists \,y\in \Phi^*(\tau, \xi) : y\in \Phi(t ,x) }
    \\
    &= \setb{x\in X: \exists \,y\in \Phi^*(\tau, \xi) : x\in \Phi^*(t ,y) } = \setb{x\in X: x\in \Phi^*(t ,\Phi^*(\tau, \xi)) }
    \\
    &= \Phi^*(t ,\Phi^*(\tau, \xi))\,.
  \end{align*}}
  This finishes the proof of this proposition.
\end{proof}

\section{Minimal invariant sets}

In the following, the focus lies on the determination and bifurcation of so-called \emph{minimal invariant sets} of a set-valued dynamical system $\Phi$. Given a set-valued dynamical system $\Phi:\T\tm X\to \cK(X)$, a nonempty and compact set $M\subset X$ is called $\Phi$-\emph{invariant} if
\begin{displaymath}
  \Phi(t,M) = M \fa t\ge0\,.
\end{displaymath}
A $\Phi$-invariant set is called \emph{minimal} if it does not contain a proper $\Phi$-invariant set.

Minimal $\Phi$-invariant sets are pairwise disjoint, and under the assumption that there exists an $\eps>0$ and $T>0$ such that $\Phi(T,x)$ contains at least an $\eps$-ball for all $x\in X$, there are only finitely many of such sets, since $X$ is compact.

Minimal $\Phi$-invariant sets are important, because they are exactly the supports of stationary measures of a random dynamical system, whenever $\Phi$ describes the topological part of the random system (see \cite{Homburg_06_1,Homburg_10_1,Botts_Unpub_1} in the continuous case of random differential equations, and \cite{Zmarrou_07_1} for random maps). Moreover, in case $\Phi$ describes a control system, minimal $\Phi$-invariant sets coincide with invariant control sets (see the monograph \cite{Colonius_00_1}).

\begin{proposition}\label{prop_3}
  Let $\Phi:\T\tm X\to \cK(X)$ be a set-valued dynamical system and let $M\subset X$ be compact with $\Phi(t,M) \subset M$ for all $t\ge0$,  and suppose that no proper subset of $M$ fulfills this property. Then $M$ is $\Phi$-invariant.
\end{proposition}

\begin{proof}
  Standard arguments lead to the fact that the $\omega$-limit set
  \begin{displaymath}
    \bigcap_{t\ge0}\overline{\bigcup_{s\ge t} \Phi(s,M)} = \bigcap_{t\ge0}\Phi(t,M)
  \end{displaymath}
  is a nonempty and compact $\Phi$-invariant set \cite{Aubin_90_1}. Since $\bigcap_{t\ge0}\Phi(t,M) \subset M$, it follows that this set coincides with $M$.
\end{proof}

The existence of minimal $\Phi$-invariant sets follows from Zorn's Lemma.

\begin{proposition}[Existence of minimal invariant sets]\label{prop_1}
  Let $\Phi:\T\tm X\to \cK(X)$ be a set-valued dynamical system and $M\subset X$ be compact such that
  $\Phi(t,M) \subset M$ for all $t\ge0$.
  Then there exists at least one subset of $M$ which is minimal $\Phi$-invariant.
\end{proposition}

\begin{proof}
  Consider the collection $\cC:= \setb{A\subset \cK(M): \Phi(t,A)\subset A  \text{ for all } t\ge 0}$.
  $\cC$ is partially ordered with respect to set inclusion, and let $\cC'$ be a totally ordered subset of $\cC$. It is obvious that $\bigcap_{A\in \cC'} A$ is nonempty,
  compact and lies in $\cC$. Thus, Zorn's Lemma implies that there exists at least one minimal element in $\cC$ which is a minimal $\Phi$-invariant set.
\end{proof}

While minimal $\Phi$-invariant sets always exists, they are typically non-unique. Uniqueness directly follows for set-valued dynamical systems that
are contractions in the Hausdorff metric. Such contractions can be identified by the following lemma, whose proof will be omitted.

\begin{lemma}
  Consider the set-valued dynamical system $\Phi:\N_0\tm \cK(X)\to\cK(X)$, defined by
  $\Phi(1,x):= U(f(x)) \fa x\in X$, where $f:X\to X$ is a contraction on the compact metric space $(X,d)$, i.e.~one has
  \begin{displaymath}
    d(f(x),f(y))\le L d(x,y) \fa x,y\in X
  \end{displaymath}
  with some Lipschitz constant $L<1$, and $U:X\to \cK(X)$ is a function such that $U(x)$ is a neighbourhood of $x$ for all $x\in X$. Assume that $U$
  is globally Lipschitz continuous (but not necessarily a contraction) with Lipschitz constant $M>0$ such that $ML <1$. The mapping $\Phi(1,\cdot)$
  then is a contraction in $(\cK(X),h)$. The unique fixed point of $\Phi(1,\cdot)$ is the unique minimal $\Phi$-invariant set, which is also globally
  attractive.
\end{lemma}


The above lemma applies in particular to the motivating example presented in the Introduction. In this case, $U(x):=\overline{B_\eps(x)}$ with
Lipschitz constant $1$. Hence, if $f$ is a contraction, then the set-valued mapping $F$ has a globally attractive unique minimal invariant set.

%
%

\section{Generalisation of attractor-repeller decomposition}\label{sec_2}

The purpose of this section is to provide generalisations of attractor-repeller decompositions which were introduced in
\cite{McGehee_06_1,Li_07_1} for the study of Morse decompositions of set-valued dynamical systems. These generalisations are necessary for our
purpose, because we deal with invariant sets rather than attractors, and they will be applied in \secref{sec_3} in the context of bifurcation theory.

Fundamental for the definition of Morse decompositions are domains of attraction (of attractors), because associated repellers are then identified as complements of these sets. For a given $\Phi$-invariant set $M$, the \emph{domain of attraction} is defined by
\begin{align*}
  \cA(M) = \setB{x\in X  : \lim_{t\to\infty} \dist\opintb{\Phi(t,x), M} = 0 }\,.
\end{align*}
If $M$ is an attractor, that is a $\Phi$-invariant set such that there exists an $\eta>0$ with
\begin{displaymath}
  \lim_{t\to\infty} \dist\opintb{\Phi(t,B_\eta(M)), M} = 0\,,
\end{displaymath}
then the complementary set $X\setminus \cA(M)$ is a $\Phi^*$-invariant set, which has the interpretation of a repeller, because all points outside of
this set converge to the attractor in forward time. It is worth to note that this repeller is not necessarily $\Phi$-invariant (which is a difference
from the classical Morse decomposition theory).

For a $\Phi$-invariant set $M$ which is not an attractor, the complementary set $X\setminus \cA(M)$ is not necessarily $\Phi^*$-invariant, but this
property can be attained when $\cA(M)$ is replaced by a slightly smaller set.

\begin{proposition}\label{prop_2}
  Let $\Phi:\T\tm X\to \cK(X)$ be a set-valued dynamical system, and let $M\subset X$ be $\Phi$-invariant such that $\cA(M) \not= X$, i.e.~$M$ is not
  globally attractive. Then the complement of the set
  \begin{align*}
    \cA_-(M):= &\cA(M) \sm \setB{x\in \cA(M) : \mbox{ there exist } t\ge 0 \mwith \Phi(t,x)\cap \partial \cA(M) \not=\es\,, \\
    & \mbox{ or for all $\gamma>0$, one has } \limsup_{t\to\infty} \dist\opintb{\Phi(t,B_\gamma(x)), M} > 0}\,,
  \end{align*}
  i.e.~the set $M^*:=X\sm \cA_-(M)$, is $\Phi^*$-invariant.
\end{proposition}

The set $M^*$ is called the \emph{dual} of $M$. Under the additional assumption that $M$ is an attractor in \propref{prop_2}, i.e.~$\cA(M)$ is a
neighbourhood of $M$, the pair $(M,M^*)$ is an attractor-repeller pair as discussed in \cite{McGehee_06_1}. This pair can be extended to obtain Morse
decompositions, see \cite{Li_07_1}.

Before proving this proposition, we will derive an alternative characterization of the set $\cA_-(M)$.

\begin{lemma}
  Let $\Phi:\T\tm X\to \cK(X)$ be a set-valued dynamical system and $M\subset X$ be $\Phi$-invariant. Then the set $\cA_-(M)$ admits the representation
  \begin{align*}
    \cA_-(M) = \setB{x\in X & : \text{for all } T \ge 0\,, \text{ there exists a neighbourhood } V \text{ of } \Phi(T,x) \\
    &\qquad \text{ with }\lim_{t\to\infty} \dist\opintb{\Phi(t,V), M} = 0 }\,.
  \end{align*}
\end{lemma}

\begin{proof}
  First, note that compact subsets $K$ of $\cA_-(M)$ are attracted by $M$, i.e.~we have $\lim_{t\to\infty} \dist(\Phi(t,K),M) = 0$. We have to show two set
  inclusions.
  \\
  $(\subset)$ Let $x\in \cA_-(M)$ and $T>0$. Since $\Phi(T,x)$ lies in the interior of $\cA(M)$, there exists a compact neighbourhood $V$ of $\Phi(T,x)$ that is contained
  in $\cA(M)$. This proves that $\lim_{t\to\infty} \dist\opintb{\Phi(t,V),M} = 0$, and hence, $x$ is contained in the right hand side.
  \\
  $(\supset)$ Let $x\in X$ such that for all $T\ge0$, there exists a neighbourhood $V$ of $\Phi(T,x)$ with $\lim_{t\to\infty} \dist
  \opintb{\Phi(t,V),M} = 0$. This implies that for all $T\ge 0$, one has $\Phi(t,x)\cap \partial \cA(M)= \es$, which finishes the proof of this
  lemma.
\end{proof}

The set $\cA_-(M)$ thus describes all trajectories in the domain of attraction that are attracted also under perturbation.

\begin{proof}[Proof of \propref{prop_2}]
  It will be shown that $\Phi^*(t,M^*)=M^*$ for all $t\ge 0$.
  \\
  $(\subset)$ Assume that there exist $t\ge0$ and $x\in \Phi^*(t,M^*)\setminus M^* = \Phi^*(t,M^*) \cap \cA_-(M)$. This implies that $\Phi(t,x) \cap M^* \not= \es$ and
  $x\in\cA_-(M)$, which contradicts the fact that $\cA_-(M)$ fulfills $\Phi(t,\cA_-(M))\subset \cA_-(M)$ for all $t\ge0$.
  \\
  $(\supset)$  Assume that there exist $t\ge 0$ and $x\in M^*\sm \Phi^*(t,M^*)$. This means that $\Phi(t, x)\cap M^*=\es$, and hence, $\Phi(t,x)\subset \cA_-(M)$. We will
  show that this implies that $x\in\cA_-(M)$, which is a contradiction. Let $T\ge 0$, and consider first the case that $T\le t$. The fact that $\cA_-(M)$ is open and
  $\Phi(t,x)\subset \cA_-(M)$ is compact implies that there exists a $\gamma>0$ such that $\overline{B_\gamma(\Phi(t,x))}\subset \cA_-(M)$. Moreover, the continuity
  of $\Phi$ and the relation $\Phi(t-T,\Phi(T,x))=\Phi(t,x)$ yield the existence of $\delta>0$ such that $\Phi\opintb{t-T, \overline{B_\delta(\Phi(T,x))}}\subset
  \overline{B_\gamma(\Phi(t,x))}\subset \cA_-(M)$. Since compact subsets of $\cA_-(M)$ are attracted to $M$, the assertion follows. Consider now the case $T\ge t$. Since
  $\cA_-(M)$ is invariant and $\Phi(t,x)\subset \cA_-(M)$, $\Phi(T,x)$ is a compact subset of $\cA_-(M)$. $\cA_-(M)$ is open, so there exists a compact neighbourhood of
  $\Phi(T,x)$ which is attracted by $M$. This finishes the proof of this proposition.
\end{proof}

\section{Dependence of minimal invariant sets on parameters}

The main goal of this section is to describe how minimal invariant sets depend on parameters. We consider a family
$(\Phi_\lambda)_{\lambda\in\Lambda}$ of set-valued dynamical systems $\Phi_\lambda:\T\tm X\to \cK(X)$, where $(\Lambda, d_\Lambda)$ is a metric space.

We assume now the conditions (H4) and (H5) that are naturally fulfilled for set-valued dynamical systems generated by mappings $f_\lambda:X\to X$, depending continuously on a real parameter $\lambda$ and perturbed by a closed $\eps$-ball as in \eqref{eqn_1}. The first condition addresses uniform continuity in $x\in X$.
\begin{itemize}
  \item[(H4)]
  $(\lambda, t)\mapsto \Phi_\lambda(t,x)$ is continuous in $(\lambda,t)\in \Lambda\times \T$ uniformly in $x$.
\end{itemize}
As in \theoref{theo_3}, we exclude single-valued dynamical systems and assume
\begin{itemize}
  \item[(H5)]
  $\Phi_\lambda(t,x)$ contains a ball of positive radius for all $(t,x)\in \T\tm X$ with $t>0$ and $\lambda\in \Lambda$, and moreover, there
  exist $T>0$ and $\eps>0$ such that $\Phi_\lambda(T,x)$ contains a ball of size $\eps$ for all $x\in X$.
\end{itemize}
%
The union of all minimal $\Phi_\lambda$-invariant sets in $X$ will be denoted by $\cM_\lambda$. The following theorem describes how $\cM_\lambda$
depends on the parameter.

\begin{theorem}[Dependence of minimal invariant sets on parameters]\label{theo_1}
  Given a family of set-valued dynamical systems $(\Phi_\lambda)_{\lambda\in\Lambda}$ satisfying (H1)--(H5), and let $M_{\lambda_\infty}\subset  \cM_{\lambda_\infty}$ be a minimal $\Phi_{\lambda_\infty}$-invariant set for some $\lambda_\infty\in\Lambda$. Then for each sequence $(\lambda_n)_{n\in\N}$ converging to $\lambda_\infty$, there exist a subsequence $(\lambda_{n_k})_{k\in\N}$ and a $\delta>0$ such that exactly one
  of the following statements holds.
  \begin{itemize}
    \item[(i)] \emph{Lower semi-continuous dependence:}
    \begin{displaymath}
      M_{\lambda_\infty} \subset \liminf_{k\to\infty} \opintb{\cM_{\lambda_{n_k}}\cap B_\delta(M_{\lambda_\infty})}\,.
    \end{displaymath}
    \item[(ii)] \emph{Instantaneous appearance:}
    \begin{displaymath}
      \es = \limsup_{k\to\infty}\opintb{\cM_{\lambda_{n_k}} \cap B_\delta(M_{\lambda_\infty})} \,.
    \end{displaymath}
  \end{itemize}
\end{theorem}

\begin{proof}
  Let $(\lambda_n)_{n\in\N}$ be a sequence with $\lambda_n\to \lambda_\infty$ as $n\to\infty$. Define the sequence $(c_n)_{n\in\N}$ by
  \begin{displaymath}
    c_n:= \left\{\begin{array}{rcc} 1 & : & \cM_{\lambda_n}\cap M_{\lambda_\infty}\not=\es\\
    2 & : & \cM_{\lambda_n}\cap M_{\lambda_\infty}=\es
    \end{array}\right. \fa n\in \N\,,
  \end{displaymath}
  and choose $\tilde \delta>0$ such that $B_{\tilde \delta}(M_{\lambda_\infty}) \cap \cM_{\lambda_\infty} = M_{\lambda_\infty}$. Since $\set{1,2}$ is finite, there exists
  a constant subsequence $(c_{n_k})_{k\in\N}$.

  If $c_{n_k}\equiv 2$, assume to the contrary that for all $k\in\N$, there exist $m\ge k$ and $a_k\in \cM_{\lambda_{n_m}} \cap
  B_{1/k}(M_{\lambda_\infty})$. The sequence $(a_k)_{k\in\N}$ has a convergent subsequence with limit $a_\infty \in M_{\lambda_\infty}$. Now $\Phi_{\lambda_\infty}(T,a_\infty)\subset M_{\lambda_\infty}$, and the continuity of $\Phi$ implies that there exists a $\gamma>0$ such that $\Phi_{\lambda_\infty}(T,x) \subset B_{\eps/4}(\Phi_{\lambda_\infty}(T,a_\infty))$ for all $x \in B_\gamma(a_\infty)$.  (H4) then implies the existence of $N>0$ such that for all $m>N$, we have $\Phi_{\lambda_{n_m}}(T,x) \subset B_{\eps/4}(\Phi_{\lambda_\infty}(T,x)) \subset B_{\eps/2}(\Phi_{\lambda_\infty}(T,a_\infty)) \subset B_{\eps/2}(M_{\lambda_\infty})$ for all $m>N$ und $x\in B_\gamma(a_\infty)$. Since $\Phi_{\lambda_{n_m}}(T,x)$ contains an $\eps$-ball and is within the $\eps/2$-neighbourhood of $M_{\lambda_\infty}$, one gets $\Phi_{\lambda_{n_m}}(T,x) \cap M_{\lambda_\infty}\not= \emptyset$. This is a contradiction to the definition of the sequence $c_{n_k}$ and this proves that there exists $\delta \in (0,\tilde\delta)$ with $\cM_{\lambda_{n_k}}\cap B_\delta(M_{\lambda_\infty})=\es$ whenever $\frac{1}{k} < \delta$. Hence, (ii) holds.

  If $c_{n_k}\equiv 1$, define $\delta:= \tilde \delta$. Choose minimal $\Phi_{\lambda_{n_k}}$-invariant sets $M_{\lambda_{n_k}}\subset\cM_{\lambda_{n_k}}$ such that $M_{\lambda_{n_k}}\cap M_{\lambda_\infty}\not=\es$ for $k\in\N$. Since $\Phi_{\lambda_\infty}(T, M_{\lambda_{n_k}}\cap M_{\lambda_\infty}) \subset M_{\lambda_\infty}$, (H4) implies that there exists a $k_0\in \N$ such that
  \begin{equation}\label{rel_3}
    \Phi_{\lambda_{n_k}}(T, M_{\lambda_{n_k}}\cap M_{\lambda_\infty})\subset B_{\eps/4}(M_{\lambda_\infty}) \cap M_{\lambda_{n_k}}\fa k\ge k_0\,.
  \end{equation}
  Let $B_{\eps/2}(d_1),\dots, B_{\eps/2}(d_r)$ with $d_1,\dots,d_r\in M_{\lambda_\infty}$ be finitely many $\eps/2$-balls covering the compact set $\overline{B_{\eps/4}(M_{\lambda_\infty})}$.  Because of (H5) and \eqref{rel_3}, each of the sets $M_{\lambda_{n_k}}$
  contains (at least) one of the points $d_1,\dots, d_r$. We can thus put the sets $M_{\lambda_{n_k}}$ into $r$ different systems of sets $\mathcal C_i$ such that $\bigcap_{M\in \mathcal C_i} M\supset \set{d_i}$ for $i\in\set{1,\dots,r}$. We show now that the asserted limit
  relation in (i) holds when restricting to a subsequence corresponding to each category, from which the assertion follows, since there are only
  finitely many categories. For simplicity, assume now that there is only one category. It will be shown now that $\liminf_{k\to\infty}
  M_{\lambda_{n_k}}$ cannot be left in forward time for $\lambda=\lambda_\infty$, i.e.~fulfills the conditions of \propref{prop_3}. Since this set is nonempty and intersects
  $M_{\lambda_\infty}$, minimality of $M_{\lambda_\infty}$ then implies the assertion. Assume to the contrary that there exists an $\tilde x\in
  \liminf_{k\to\infty} M_{\lambda_{n_k}}$ such that $\Phi_{\lambda_\infty}(\tau,\tilde x)\sm \liminf_{k\to\infty} M_{\lambda_{n_k}} \not= \es$ for
  some $\tau>0$, i.e.~there exists a $\xi\in \Phi_{\lambda_\infty}(\tau,\tilde x)$ such that $ \xi \notin \liminf_{k\to\infty} M_{\lambda_{n_k}}$.
  (H4) implies the existence of a sequences $(x_{n_k})_{k\in\N}$ (converging to $\tilde x$) and $(y_{n_k})_{k\in\N}$ (converging to $\xi$) such that $x_{n_k}\in M_{\lambda_{n_k}}$ and $y_{n_k}\in \Phi_{\lambda_{n_k}}(\tau,x_{n_k})\subset M_{\lambda_{n_k}}$. Hence, $\xi\in \liminf_{k\to\infty}  M_{\lambda_{n_k}}$, which is a contradiction and finishes the proof of this theorem.
\end{proof}

The above theorem asserts that discontinuous changes in minimal invariant sets occur either as explosions or as instantaneous appearances. We are led to address the question if a continuous merging process of two minimal invariant sets is possible (note that this is not ruled out by (i) of
\theoref{theo_1}). The following proposition shows that the answer to this question is negative.

\begin{proposition}\label{prop_4}
  Let $(\Phi_\lambda)_{\lambda\in\Lambda}$ be a family of set-valued dynamical systems fulfilling (H1)--(H5), and let $M_\lambda^1$ and   $M_\lambda^2$ be two different minimal $\Phi_\lambda$-invariant sets. Then for all $\lambda^*\in \Lambda$, one has
  \begin{displaymath}
    \liminf_{\lambda\to\lambda^*} \inf_{(x,y)\in M_\lambda^1\tm M_\lambda^2} d(x,y) > 0\,,
  \end{displaymath}
  i.e.~the sets $M_\lambda^1$ and $M_\lambda^2$ cannot collide under variation of $\lambda$.
\end{proposition}

\begin{proof}
  Suppose the contrary, which means that there exist an $x^*\in X$ and a sequence $\lambda_n\to\lambda^*$ as $n\to\infty$ with
  \begin{displaymath}
    \lim_{n\to\infty} \dist(x^*, M^1_{\lambda_n}) = 0 \qandq \lim_{n\to\infty} \dist(x^*, M^2_{\lambda_n}) = 0\,.
  \end{displaymath}
  Due to (H4) and (H5), for $t>0$, the set $\Phi_{\lambda^*}(t, x^*)$ intersects the interior of both $M^1_{\lambda_n}$ and $M^2_{\lambda_n}$ when
  $n$ is large enough. This, however, contradicts the fact that $M^1_{\lambda_n}$ and $M^2_{\lambda_n}$ are $\Phi$-invariant and finishes the
  proof of this proposition.
\end{proof}

\section{A necessary condition for bifurcation}\label{sec_3}

Consider a family $(\Phi_\lambda)_{\lambda\in\Lambda}$ of set-valued dynamical systems $\Phi_\lambda:\T\tm X\to \cK(X)$, where $(\Lambda, d_\Lambda)$
is a metric space, and suppose that (H1)--(H5) hold.

Recall the definition of a topological bifurcation (\defref{def_1}) and the fact that $\cM_\lambda$ denotes the union of all minimal
$\Phi_\lambda$-invariant sets. As a direct consequence of \theoref{theo_1} and \propref{prop_4}, a
topological bifurcation of $\cM_\lambda$ is characterised by a minimal $\Phi_{\lambda_\infty}$-invariant set $M_{\lambda_\infty}$, a sequence
$\lambda_n\to \lambda_\infty$ as $n\to\infty$ and $\delta>0$ such that
\begin{equation}\label{rel_2}
  M_{\lambda_\infty} \subsetneq \liminf_{n\to\infty} \opintb{\cM_{\lambda_n}\cap B_\delta(M_{\lambda_\infty})} \qorq
  \es = \limsup_{n\to\infty}\opintb{\cM_{\lambda_n} \cap B_\delta(M_{\lambda_\infty})}\,.
\end{equation}
The following theorem provides a necessary condition for a topological bifurcation of minimal invariant sets involving the dual
$M_{\lambda_\infty}^*$ of $M_{\lambda_\infty}$ as introduced in \secref{sec_2}.

\begin{theorem}[Necessary condition for bifurcation]\label{theo_2}
  Let $(\Phi_\lambda)_{\lambda\in\Lambda}$ be a family of set-valued dynamical systems fulfilling (H1)--(H5), and assume that
  $(\Phi_\lambda)_{\lambda\in\Lambda}$ admits a topological bifurcation such that \eqref{rel_2} holds for a minimal invariant set
  $M_{\lambda_\infty}$. Then $M_{\lambda_\infty}^*$ has nonempty intersection with $M_{\lambda_\infty}$.
\end{theorem}

\begin{proof}
  Consider the sequence $\lambda_n\to\lambda_\infty$ as defined before the statement of the theorem. Assume to the contrary that there exists a
  $\gamma>0$ such that $B_\gamma(M_{\lambda_\infty})\subset \cA_-(M_{\lambda_\infty})$. Then for each $\delta>0$, there exists a compact absorbing
  set $B$ such that $M_{\lambda_\infty}\subset B\subset B_{\delta}(M_{\lambda_\infty})$, that is, $\Phi_{\lambda_\infty}(t,B)\subset \inner B$ for
  $t>0$. 
  Due to continuous dependence on $\lambda$, there exists an $n_0\in\N$ such that
  $\Phi_{\lambda_n}(t,B)\subset \inner B$ for all $n\ge n_0$ and $t>0$. This means that there exists a minimal $\Phi_{\lambda_n}$-invariant set in
  $B$ for all $n\ge n_0$. Note that $n_0$ depends on $\delta$, and in the limit $\delta \to 0$, this minimal invariant set converges to
  $M_{\lambda_\infty}$, because of \theoref{theo_1}. Hence, there is no bifurcation, which shows that $X\sm \cA_-(M_{\lambda_\infty})\cap
  M_{\lambda_\infty} \not= \es$.
\end{proof}

\providecommand{\bysame}{\leavevmode\hbox to3em{\hrulefill}\thinspace}
\providecommand{\MR}{\relax\ifhmode\unskip\space\fi MR }
\providecommand{\MRhref}[2]{%
  \href{http://www.ams.org/mathscinet-getitem?mr=#1}{#2}
}
\providecommand{\href}[2]{#2}

\end{document}